\documentclass[letterpaper, 10 pt, conference]{ieeeconf}  

\IEEEoverridecommandlockouts                              

\usepackage{amsthm}

\theoremstyle{plain}
\newtheorem{theorem}{Theorem}
\newtheorem{corollary}{Corollary}

\newtheorem{lemma}{Lemma}

\theoremstyle{definition}
\newtheorem{definition}{Definition}

\newtheorem{example}{Example}
\theoremstyle{plain}

\theoremstyle{remark}
\newtheorem{remark}{Remark}

\newcommand{\mE}{\ensuremath{\mathbb{E}}}
\newcommand{\mF}{\ensuremath{\mathbf{F}}}

\newcommand{\mP}{\ensuremath{\mathbb{P}}}

\newcommand\gC{{\mathcal{C}}}

\newcommand\gE{{\mathcal{E}}}

\newcommand\gL{{\mathcal{L}}}

\newcommand\gX{{\mathcal{X}}}
\newcommand\gY{{\mathcal{Y}}}

\providecommand{\keywords}[1]{\vspace*{0.5\baselineskip}\par\noindent\textbf{Keywords:} #1}

{%
\begin{list}{#1}{
\vspace{-\topsep}
\vspace{-\partopsep}
\setlength{\itemindent}{0cm}
\setlength{\rightmargin}{0cm}
\setlength{\listparindent}{0cm}
\settowidth{\labelwidth}{#1}
\setlength{\leftmargin}{\labelwidth}
\addtolength{\leftmargin}{\labelsep}
\setlength{\itemsep}{0cm}
}%
}%
{%
\end{list}
\vspace{-\topsep}
\vspace{-\partopsep}
}

{\begin{enumerate}%
\renewcommand{\theenumi}{\roman{enumi}}%
}%
{\end{enumerate}}

\usepackage{bbold}

\usepackage[sort, nocompress]{cite}
\usepackage{amsmath,amssymb,amsthm, amsfonts, mathrsfs}
\usepackage{graphicx}
\usepackage{xcolor}
\usepackage{comment}

\usepackage{cleveref}
\Crefname{equation}{}{}
\crefname{equation}{}{}
\usepackage{footmisc}
\usepackage{csquotes}
\let\labelindent\relax
\usepackage{enumitem}

\usepackage{algorithm}
\usepackage{algpseudocode}

\usepackage{caption, subcaption}

\def\BibTeX{{\rm B\kern-.05em{\sc i\kern-.025em b}\kern-.08em
    T\kern-.1667em\lower.7ex\hbox{E}\kern-.125emX}}

\newcommand{\ie}{\textit{i.e., }}
\newcommand{\eg}{\textit{e.g., }}
    
\newcommand{\R}{\mathbb R}    
   
\newcommand\eqnumber{\addtocounter{equation}{1}\tag{\theequation}}

\DeclareMathOperator*{\argmin}{arg\,min}
\newcommand{\interior}{\text{int}\,}
\newcommand{\C}{\mathcal{C}}  

\newcommand{\exit}{\Gamma}
\newcommand{\minf}{\Phi}
\newcommand{\entrance}{\Psi}
\newcommand{\maxf}{\Theta}

\newcommand{\sprob}{\mF}
\newcommand{\varsprob}{\tilde{\mF}}

\newcommand{\dimw}{\omega}

\newcommand{\Hess}[1]{\operatorname{Hess}{#1}}
\newcommand{\grad}[1]{\nabla{#1}}
\newcommand{\tr}{\operatorname{tr}}

\newcommand{\addmodfunc}{\kappa}

\newcommand{\varmargin}{L}

\title{\LARGE \bf
Myopically Verifiable Probabilistic Certificate for Long-term Safety
}

\author{Zhuoyuan Wang$^{1\dagger}$, Haoming Jing$^{1\dagger}$, Christian Kurniawan$^{1}$, Albert Chern$^{2}$, Yorie Nakahira$^{1*}$
\thanks{*This work is supported by JST, PRESTO Grant Number JPMJPR2136, Japan.}
\thanks{$^\dagger$These authors contributed equally.}
\thanks{$^{1}$Zhuoyuan Wang, Haoming Jing, Christian Kurniawan and Yorie Nakahira are with the Department of Electrical and Computer Engineering, Carnegie Mellon Universty,
        {\tt\small \{zhuoyuaw,haomingj,ckurniaw,
        ynakahir\}@andrew.cmu.edu}.}%
\thanks{$^{2}$Albert Chern is with the Department of Computer Science and Engineering, University of California San Diego, {\tt\small alchern@ucsd.edu}.}%
\thanks{$*$To whom correspondence should be addressed.}
}

\begin{document}

\maketitle
\thispagestyle{empty}
\pagestyle{empty}

\begin{abstract}

In this paper, we consider the use of barrier function-based approaches for the safe control problem in stochastic systems. With the presence of stochastic uncertainties, a myopic controller that ensures safe probability in infinitesimal time intervals may allow the accumulation of unsafe probability over time and result in a small long-term safe probability. Meanwhile, increasing the outlook time horizon may lead to significant computation burdens and delayed reactions, which also compromises safety. To tackle this challenge, we define a new notion of forward invariance on ‘probability space’ as opposed to the safe regions on state space. This new notion allows the long-term safe probability to be framed into a forward invariance condition, which can be efficiently evaluated. We build upon this safety condition to propose a controller that works myopically yet can guarantee long-term safe probability or fast recovery probability. The proposed controller ensures the safe probability does not decrease over time and allows the designers to directly specify safe probability. The performance of the proposed controller is then evaluated in numerical simulations. Finally, we show that this framework can also be adapted to characterize the speed and probability of forward convergent behaviors, which can be of use to finite-time Lyapunov analysis in stochastic systems. 


\end{abstract}

\section{Introduction}
\label{S:Introduction}

Autonomous systems (\eg robots and self-driving vehicles) must make safe control decisions in real-time and in the presence of various uncertainties. The control of such safety- and delay-critical systems relies extensively on barrier function-based approaches. Barrier function-based approaches can provide provable safety with low computation cost within deterministic systems that possess small and bounded noise due to the two features stated below~\cite{ames2019control,khatib1986real,prajna2007framework}: computation efficiency arising from a myopic controller (feature 1) and from the use of analytical/affine safety conditions (feature 2). However, these two features did not necessarily translate to stochastic systems whose uncertainty is captured by random variables with unbounded support, as we will discuss below. In this paper, we overcome this difficulty by characterizing a sufficient condition for ‘invariance’ in the probability space. This condition is then used to guarantee the unsafe probability to below the tolerable levels without the loss of these two features. 

\textit{Feature 1: Computation efficiency arising from a myopic controller.} In a deterministic system, safety can be guaranteed if the state never moves outside the safe set within an infinitesimal outlook time interval. This property allows a myopic controller, which only evaluates the infinitesimal outlook time interval (immediate future time), to keep the system safe at all times. A myopic evaluation requires much less computation than methods that evaluate a long time horizon since the computational load to evaluate possible future trajectories significantly increases with the outlook time horizon. 

In a stochastic whose uncertainty has unbounded support, however, the probability of staying within the safe set in the infinitesimal outlook time interval is strictly less than one. In other words, there will always be a non-zero tail probability to move outside of the safe set. This tail probability can accumulate over time and result in a small long-term safe probability. This suggests the need for a more refined \textit{temporal} characterization of long-term safe/unsafe probabilities. 

\textit{Feature 2: Computation efficiency arising from the use of analytical/affine safety conditions.} In a deterministic system, the condition for the state to stay within the safe set in an infinitesimal time can be translated as requiring the vector field of the state stays within the tangent cone of the safe set~\cite{Blanchini1999}. A sufficient condition of this requirement is expressed using analytic inequalities that are affine in the control action and thereby can be can be integrated into quadratic programs (see~\cite{ames2016control} and references therein). 

In a stochastic system, however, constraining the mean trajectory to satisfy this condition, without bounding the higher moments, does not give us control over the tail probability of the state moving outside of the safe region. This suggests the need for a more refined \textit{spatial} characterization of unsafe behaviors and state distribution. 

Therefore, ensuring safety in a stochastic system needs more refined temporal and spatial characterization of safe/unsafe behaviors during a long outlook time interval. However, the former requires tracing the long-term evolution of complex dynamics, environmental changes, control actions, as well as their couplings. While the latter requires characterization of the state distribution, tails, and conditional value at risk. Both compromise the above two features and can impose a significant computational burden. Such heavy computation can compromises safety due to slower response, despite the use of more optimized actions.

Prior work has yielded diverse approaches for finer time/space characterization in stochastic systems, but all wrestle with this important safety/reaction time tradeoff. We approximately classify these approaches into three main types based on their choice of tradeoffs: long-term safety with heavy computation (approach A); myopic safety with low computation (approach B) ; and long-term conservative safety with low computation (approach C).

\subsection{Related Work}

\textit{Approach A: long-term safety with heavy computation.} There exists extensive literature that considers a long time horizon and/or the state distribution (or higher moments of the state distribution) at the expense of high computation costs. For example, various model predictive control (MPC) and chance-constrained optimization include safety constraints in a long time horizon (see~\cite{Farina2016,Hewing2020} and references therein). Reachability-based techniques use the characterization of reachable states over a finite/infinite time horizon to constrain the control action so that the state reaches or avoids certain regions~\cite{Chen2018}. Within barrier function-based approaches, the safety condition can be formulated as constraints on the control action that involve the conditional value-at-risk (CVaR) of the barrier function values~\cite{ahmadi2020risk}. While these techniques can find more optimal control actions that are safe in the long term, they often come with significant computation costs. The cause is twofold: first, possible trajectories often scale exponentially with the length of the outlook time horizon; and second, tails or CVaR involve the probability and mean of rare events, which are more challenging to estimate than nominal events. Such stringent tradeoffs between estimating longer-term safe probability vs. computation burden limit the utility of these techniques in delay-critical systems for more expansive (longer time scale or precise characterization of the state distribution) control action evaluation.

\subsubsection*{Approach B: myopic safety with low computation.} Motivated by the latency requirement in real-time safety-critical control, a few approaches use myopic controllers that constrain the probability of unsafe events in an infinitesimal time interval. For example, the stochastic control barrier function use a sufficient condition for ensuring that the state, on average, moves within the tangent cone of the safe set~\cite{clark2019control}. The probabilistic barrier certificate ensures certain conditions of the barrier functions to be satisfied with high probability~\cite{luo2019multi}. The myopic nature of these methods achieves a significant reduction in computational cost but can result in unsafe behaviors in a longer time horizon due to the accumulation of tail probabilities of unsafe events.

\subsubsection*{Approach C: long-term conservative safety with low computation.} To have a faster response but still achieve longer-term safety, other approaches use probability and/or martingale inequalities to derive sufficient conditions for constraining the evolution of barrier function values in a given time interval~\cite{prajna2007framework,yaghoubi2020risk,santoyo2021barrier}. These sufficient conditions are given analytically and are elegantly integrated into the convex optimization problems to synthesize controllers offline or verify control actions online. The controllers based on these techniques often require less online computation to find the action that guarantees longer-term safety. However, due to the approximate nature of the probabilistic inequalities, the control actions can be conservative and unnecessarily compromise nominal performance.

\subsection*{Contribution of this paper}

In this paper, we propose an efficient algorithm that ensures safety during a fixed or receding time horizon. The algorithm is based on a new safety condition that is sufficient to control the unsafe probability in a given time interval to stay above the tolerable risk levels.\footnote{Here, we consider two types of unsafe probability: the probability of exiting the safe set in a time interval when originated inside and the probability of recovering to the safe set when originated outside.} This safety condition is constructed by translating probabilistic safety specifications into a forward invariance condition on the level sets of the safe probability. The use of forward invariance allows safety at all time points to be guaranteed by a myopic controller that only evaluates the state evolution in an infinitesimal future time interval. Moreover, the sufficient condition is affine to the control action and can be used in convex/quadratic programs. The parameters of the sufficient condition are determined from the safe probability, its gradient, and its hessian. These values satisfy certain deterministic convection-diffusion equations (CDEs), which characterize the boundary conditions and the relationship between the safe probabilities of neighboring initial conditions and time horizons. These CDEs can be combined with the Monte Carlo (MC) method to improve the accuracy and efficiency in computing these values.

Below, we summarize the advantages of the proposed algorithms. 

\subsubsection*{Advantage 1} Computation efficiency. The proposed method only myopically evaluates the immediate future using closed-form safety constraints. Thus, it can have reduced computational burdens than approach A. 

\subsubsection*{Advantage 2} Provable guarantee in long-term safe probability. The closed-form safety constraints are derived from the safe probability during a receding or fixed time horizon. Thus, the proposed method can have more direct control over the probability of accumulating tail events than approach B. 

\subsubsection*{Advantage 3} Intuitive parameter tuning using exact safety vs. performance tradeoffs. The proposed method uses exact characterizations of safe probability. Thus, it allows the aggressiveness towards safety to be directly tuned based on the exact probability, as opposed to probabilistic bounds or martingale approximations used in approach C. Moreover, our framework may be useful in characterizing the speed and probability of forward convergence in finite-time Lyapunov analysis of stochastic systems.

\section{Preliminary}
\label{S:Preliminary}

Let $\R$, $\R_+$, $\R^n$, and $\R^{m\times n}$ be the set of real numbers, the set of non-negative real numbers, the set of $n$-dimensional real vectors, and the set of $m \times n$ real matrices, respectively. Let $x[k]$ be the $k$-th element of vector $x$. Let $f:\gX \rightarrow \gY$ represent that $f$ is a mapping from space $\gX$ to space $\gY$. Let \(\mathbb{1}\{\gE\}\) be an indicator function, which takes \(1\) when condition \(\gE\) holds and \(0\) otherwise. Let $\mathbf{0}_{m\times n}$ be an $m\times n$ matrix with all entries $0$. Given events $\gE$ and $\gE_c$, let $\mP(\gE)$ be the probability of $\gE$ and $\mP(\gE | \gE_c)$ be the conditional probability of $\gE$ given the occurrence of $\gE_c$. Given random variables $X$ and $Y$, let $\mE[X]$ be the expectation of $X$ and $\mE[X | Y = y]$ be the conditional expectation of $X$ given $Y=y$. We use upper-case letters (\eg $Y$) to denote random variables and lower-case letters (\eg $y$) to denote their specific realizations.

\begin{definition}[Infinitesimal Generator]
    The infinitesimal generator $A$ of a stochastic process $\{Y_t\in\R^n\}_{t\in\R_+}$ is
    \begin{align}
    \label{eq:afy}
        AF(y)=\lim_{h\to 0}\frac{\mE[F(Y_h)| Y_0 = y]-F(y)}{h}
    \end{align}
    whose domain is the set of all functions $F:\R^n\rightarrow\R$ such that the limit of \eqref{eq:afy} exists for all $y\in\R^n$.
\end{definition}

\section{Problem Statement}
\label{S:Problem_Statement}

Here, we introduce the control system in subsection~\ref{sec:SystemDescription}, define the measures to characterize two types of safety in subsection~\ref{SS:Characterization of Safe Behaviours}, and state the controller design goals in subsection~\ref{SS:Safety_Specification}. 

\subsection{Control System Description} 
\label{sec:SystemDescription}
 
We consider a time-invariant stochastic control and dynamical system. The system dynamics is given by the stochastic differential equation (SDE)
\begin{align}
\label{eq:x_trajectory}
    dX_t = \left(f(X_t) + g(X_t)U_t\right) dt + \sigma(X_t) dW_t,
\end{align}
where $X_t \in \R^n$ is the system state, $U_t \in \R^m$ is the control input, and $W_t \in \R^\dimw$ captures the system uncertainties. Here, $X_t$ can include both the controllable states of the system and the uncontrollable environmental variables such as moving obstacles. We assume that $W_t$ is the standard Wiener process with $0$ initial value, \ie $W_0=0$. The value of $\sigma(X_t)$ is determined based on the size of uncertainty in unmodeled dynamics and environmental variables. 

The control action $U_t$ is determined at each time by the control policy. We assume that accurate information of the system state can be used for control. The control policy is composed of a nominal controller and additional modification scheme to ensure the safety specifications illustrated in subsections~\ref{SS:Characterization of Safe Behaviours} and~\ref{SS:Safety_Specification}. The nominal controller is represented by 
\begin{align}
\label{eq:nominal_controller}
    U_t = N(X_t),
\end{align}
which does not necessarily account for the safety specifications defined below. To adhere to the safety specifications, the output of the nominal controller is then modified by another scheme. The overall control policy involving the nominal controller and the modification scheme is represented by 
\begin{align}
\label{eq:generic_controller}
    U_t = K_N(X_t, L_t, T_t),
\end{align}
where $K_N: \R^{n}\times\R\times\R \rightarrow \R^m$ is a deterministic function of the current state $X_t$, safety margin $L_t$, and time horizon $T_t$ to the current control action $U_t$. The policy of the form~\eqref{eq:generic_controller} assumes that the decision rule is time-invariant,\footnote{\label{ft:timeinvariant-control} The functions $N$, $K_N$ do not change over time} and that the control action can be uniquely determined for each $(X_t, L_t, T_t)$. This policy is also assumed to be memory-less in the sense that it does not use the past history of the state $\{X_\tau\}_{\tau < t}$ to produce the control action $U_t$. The assumption for memory-less controller is reasonable because the state evolution $dX_t$ of system \eqref{eq:x_trajectory} only depends on the current system state $X_t$.\footnote{\label{ft:timeinvariant} Note that $f(X_t)$, $g(X_t)$, and $\sigma(X_t)$ are time-invariant functions of the system state.} We restrict ourselves to the settings when $f$, $g$, $\sigma$, $N$, and $K_N$ have sufficient regularity conditions such that both the closed loop system of \eqref{eq:x_trajectory} and \eqref{eq:generic_controller} have unique strong solutions.\footnote{\label{ft:FN-Solution}Conditions required to have a unique strong solution can be found in~\cite[Chapter~5]{moustris2011evolution},~\cite[Chapter~1]{oksendal_stochastic_2003a}, \cite[Chapter II.7]{borodin_stochastic_2017} and references therein.}

The safe region of the state is specified by the zero super level set of a continuously differentiable barrier function $\phi(x) : \R^n \rightarrow \R$, \ie 
\begin{align}
    \gC(0) = \left\{x \in \R^n : \phi(x) \geq 0 \right\}.
\end{align}
We use 
\begin{align}
\label{eq:l-level_set}
    \gC(L) := \left\{x \in \mathbb{R}^{n}: \phi(x) \geq L \right\}
\end{align}
to denote the set with safety margin $L$. 
Accordingly, we use $\interior{\gC(0)}  = \{x\in\R^{n}: \phi(x) > 0  \}$ to denote the interior of the safe set, $\gC(0)^c = \{x\in\R^{n}: \phi(x)<0 \}$ to denote the unsafe set, $\partial \gC(L) = \{x\in\R^n:\phi(x) = L \}$ to denote the boundary of $L$ super level set. 


\subsection{Probabilistic Characterization of Safe Behaviours}
\label{SS:Characterization of Safe Behaviours}

The system must satisfy the following two types of probabilistic safety specifications: forward invariance and forward convergence.

\subsubsection{Forward Invariance}
\label{SSS:Forward_Invariance}

The forward invariance property refers to the system's ability to keep its state within a set when the state originated from the set. The probabilistic forward invariance to a set $\gC(L)$ can be quantified using 
\begin{align}
\label{eq:forward_invariance}
    \mP\left(\, X_\tau \in \gC(L), \forall \tau \in [t,t+T] \ | \ X_t=x \, \right)
\end{align}
for some time interval $[0,T]$ conditioned on an initial condition $x \in \gC(L)$. Probability \cref{eq:forward_invariance} can be computed from the distribution of the following two random variables:\footnote{\label{eq:rv-intro} These random variables are previously introduced and analyzed in~\cite{chern2021safecontrolatcdc}.}
\begin{align}
\label{eq:min_phi}
   \minf_x(T) &:= \inf\{ \phi(X_t) \in \R : t\in [0,T] , X_0 = x  \},\\
\label{eq:first_exit_time}
   \exit_x(L)  &:= \inf\{t \in \R_+ :  \phi(X_t) < L , X_0 = x \}.
\end{align}
Here, $\minf_x(T)$ is the worst-case safety margin from the boundary of the safe set $\partial \gC(0)$ during $[0,T]$, and $\exit_x(L)$ is the time when the system exit from $\gC(L)$ for the first time. We can rewrite \cref{eq:forward_invariance} using the two random variables \cref{eq:min_phi,eq:first_exit_time} as
\begin{align} 
    &\mP\left(\, X_\tau \in \gC(L), \forall \tau \in [t,t+T] \ | \ X_t=x \, \right)\\
    \label{eq:safe_prob_invariant}
        &=\: \mP\left(X_\tau \in \gC(L), \forall \tau \in [0,T] \ | \ X_0=x \right) \\
        &=\: \mP (\minf_x(T) \geq \varmargin) \\
        &=\: \mP(\exit_x(\varmargin) > T ) =\: 1- \mP(\exit_x(\varmargin) \leq T).
\end{align}
Here, equality \eqref{eq:safe_prob_invariant} holds due to the time-invariant nature of the system\footref{ft:timeinvariant} and control policies\footref{ft:timeinvariant-control}. 

\subsubsection{Forward Convergence}
\label{SSS:Safety_Recovery}

The forward convergence property indicates the system's capability for its state to enter a set when the state originated from outside the set. This probabilistic forward convergence can be quantified using
\begin{align}
\label{eq:forward_convergence}
    \mP\left(\, \exists \tau \in [t,t+T] \text{ s.t. } X_\tau \in \gC(L)\ |\ X_t = x \, \right)
\end{align}
for some time interval $[0,T]$ conditioned on an initial condition $x \in \gC(L)^c $. Similar to the case of forward invariance, probability \cref{eq:forward_convergence} can also be computed from the distribution of the following two random variables:\footref{eq:rv-intro}
\begin{align}
\label{eq:max_phi}
   \maxf_x(T) & := \sup\{ \phi(X_t) \in \R : t \in [0,T] , X_0 = x  \},\\
\label{eq:first_entry_time}
   \entrance_x(L) & := \inf\{t \in \R_+ :  \phi(X_t)\geq L , X_0 = x \}.
\end{align}
Here, $\maxf_x(T)$ indicates the distance to the boundary of the safe set $\partial\gC(0)$, and $\entrance_x(L)$ is the duration for the state to enter the set $\gC(L)$ for the first time. We can also rewrite \cref{eq:forward_convergence} using the two random variables \cref{eq:max_phi,eq:first_entry_time} as
\begin{align} 
\label{eq:safe_prob_convergent}
        &\mP\left(\, \exists \tau \in [t,t+T] \text{ s.t. } X_\tau \in \gC(L)\ |\ X_t = x \, \right)\\
        &=\: \mP\left(\ \exists \tau \in [0,T] \text{ s.t. } X_\tau \in \gC(L)\ |\ X_0 = x \ \right)\\
        & =\:  \mP( \maxf_x(T) \geq \varmargin) \\
        & =\:  \mP ( \entrance_x(\varmargin) \leq  T).
\end{align}


\subsection{Design goals}
\label{SS:Safety_Specification}

In this paper, we design the control policy with the long-term safety guarantees given in the forms alike \eqref{eq:forward_invariance} or \eqref{eq:forward_convergence}.

When the goal is to guarantee probabilistic forward invariance, we aim to ensure the following condition: for each time $t \in \R_+$,
\begin{align}
\label{eq:FI-probability}
    \mP \big( X_\tau \in\gC(\varmargin_t), \forall \tau\in [t,t+T_t] \big) \geq 1-\epsilon, 
\end{align}
conditioned on the initial condition $X_0 = x$, for some $\epsilon \in (0,1)$. From now on, all probabilities are conditioned on the initial condition $X_0 = x$ unless otherwise noted. Here, $\varmargin_t$ is the desired safety margin, and $T_t$ is the outlook time horizon. For each time $t$, condition \eqref{eq:FI-probability} constrains the probability of staying within the safe set with margin $L_t$ during the time interval $[t, t+T_t]$ to be above $1-\epsilon$.

When the goal is to guarantee probabilistic forward convergence, we aim to ensure the following condition: for each time $t \in \R_+$,
\begin{align}
\label{eq:FC-probability}
    \mP \big(\exists \tau \in [t,t+T_t] \text{ s.t. } X_\tau \in \gC(\varmargin_t) \big)\geq 1-\epsilon,
\end{align}
conditioned on the initial condition $X_0 = x$, for some $\epsilon \in (0,1)$. 

In both cases, the value of $\epsilon \in (0,1)$ is chosen based on risk tolerance. In \eqref{eq:FI-probability} and \eqref{eq:FC-probability}, the probabilities are taken over the distribution of $X_t$ and its future trajectories $\{X_\tau\}_{\tau \in (t, t+T_t]}$ conditioned on $X_0=x$. 
The distribution of $X_t$ is generated based on the closed-loop system of \eqref{eq:x_trajectory} and \eqref{eq:generic_controller}, whereas the distribution of $\{X_t\}_{t \in (t, t+ T_t]}$ are allowed to be defined in two different ways based on the design choice: the closed-loop system of \eqref{eq:x_trajectory} and \eqref{eq:nominal_controller} or the closed-loop system of \eqref{eq:x_trajectory} and \eqref{eq:generic_controller}.

We consider either fixed time horizon or receding time horizon. In the fixed time horizon, safety is evaluated at each time $t$ for a time interval $[t,t+H]$ of fixed length. In the receding time horizon, we evaluate, at each time $t$, safety only for the remaining time $[t,H]$ given a fixed horizon. The outlook time horizon for each case is given by
\begin{align}
\label{eq:time_horizon_cases}
    T_t & =
    \begin{cases}
        H, & \text{for fixed time horizon,} \\
        H-t, & \text{for receding time horizon.}
    \end{cases}
\end{align}
The safety margin is assumed to be either fixed or time varying. Fixed margin refers to when the margin remains constant at all time, \ie $\varmargin_t = \ell$. For time-varying margin, we consider the margin $\varmargin_t$ that evolves according to
\begin{align}
\label{eq:f_ell_def}
    d\varmargin_t = f_\ell (\varmargin_t),\ \varmargin_0 = \ell,
\end{align}
for some continuously differentiable function $f_\ell$.\footnote{This representation also captures fixed margin by setting $f_\ell (\varmargin_t) \equiv 0$.} The values of $T_t $ and $\{\varmargin_t\}_{t \in [0,\infty)}$ are determined based on the design choice.

%




\section{Proposed Method}
\label{S:Proposed_Method}

Here, we present a sufficient condition to achieve the safety requirements in subsection \ref{SS:safety_conditions}. Based on this condition, we propose two safe control algorithms in subsection \ref{SS:Proposed_Algorithm} and outline a method to boost algorithm performance in subsection \ref{SS:Accurate_Gradient}.

Before presenting these results, we first define a few notations. To capture the time-varying nature of $T_t \text{ and } \varmargin_t$, we augment the state space as
\begin{align}
    Z_t:=\begin{bmatrix}
    T_t \\
    \varmargin_t \\
    \phi(X_t)\\
    X_t
    \end{bmatrix}\in\R^{n+3}.
\end{align}
The dynamics of $Z_t$ satisfies the following SDE:
\begin{align}
\label{eq:z_t_dyn}
    dZ_t = (\tilde{f}(Z_t)+\tilde{g}(Z_t)U_t) dt+\tilde{\sigma}(Z_t)dW_t.
\end{align}
Here, $\tilde{f}$, $\tilde{g}$, and $\tilde{\sigma}$ are defined to be
\begin{align}
\label{eq:tilde_f}
    \tilde{f}(Z_t) & :=\begin{bmatrix}
    f_T\\
    f_\ell (\varmargin_t)\\
    f_\phi (X_t)\\
    f(X_t)
    \end{bmatrix} \in \R^{(n+3)}, \\
    \tilde{g}(Z_t) & :=\begin{bmatrix}
    \mathbf{0}_{2\times n}\\
    \gL_{g}\phi(X_t)  \\
    g(X_t)
    \end{bmatrix} \in \R^{(n+3) \times m}, \\
    \tilde{\sigma}(Z_t) & :=\begin{bmatrix}
    \mathbf{0}_{2\times n}\\
    \gL_\sigma \phi(X_t) \\
    \sigma(X_t)
    \end{bmatrix} \in \R^{(n+3) \times \dimw}.
\end{align}
In \eqref{eq:tilde_f}, the scalar $f_T$ is given by
\begin{align}
\label{eq:f_T_def}
f_T&:=\begin{cases}
        0, & \text{in fixed time horizon,} \\
        -1, & \text{in receding time horizon,}
    \end{cases}
\end{align}
the function $f_\ell$ is given by \eqref{eq:f_ell_def}, and the function $f_\phi$ is given by 
\begin{align}
\label{eq:f_phi_def}
f_\phi (X_t)&:= \gL_{f}\phi(X_t) +\frac{1}{2} \text{tr} \big(\left[\sigma(X_t)\right]\left[\sigma(X_t)\right]^\intercal\Hess \phi(X_t)\big).
\end{align}

\begin{remark}
\label{rm:lie-derivative}
The Lie derivative of a function $\phi(x)$ along the vector field $f(x)$ is denoted as $\gL_{f}\phi(x) = f(x) \cdot \nabla \phi(x) $. The Lie derivative $(\gL_{g} \phi(x))$ along a matrix field $g(x)$ is interpreted as a row vector such that $\left( \gL_{g}\phi(x) \right) u  = \left( g(x) u \right) \cdot \nabla \phi(x)$.
\end{remark}


\subsection{Conditions to Assure Safety}
\label{SS:safety_conditions}

We consider one of the following four types of probabilistic quantities:\footnote{Recall from Section~\ref{SS:Safety_Specification} that whenever we take the probabilities (and expectations) over paths, we assume that the probabilities are conditioned on the initial condition \(X_0 = x\).}
\begin{align}
\label{eq:cases_summary}
    \sprob(Z_t) := 
    \begin{cases}
        \mP \left(\minf_{X_t}(T_t) \geq \varmargin_t \right) & \text{for type I,} \\
        \mP \left(\exit_{X_t}(\varmargin_t) > T_t \right) & \text{for type II,} \\
        \mP \left( \maxf_{X_t}(T_t) \geq \varmargin_t \right) & \text{for type III,} \\
        \mP \left( \entrance_{X_t}(\varmargin_t) \leq  T_t \right) & \text{for type IV,}
    \end{cases}
\end{align}
where the probability is taken over the same distributions of $\{X_\tau\}_{\tau \in [t, T_t]}$ that are used in the safety requirement \cref{eq:FI-probability,eq:FC-probability}. The values of $T_t$ and $\varmargin_t$ (known and deterministic) are defined in \cref{eq:time_horizon_cases,eq:f_ell_def} depending on the design choice of receding/fixed time-horizon and fixed/varying margin. 

Additionally, we define the mapping $D_\sprob:\R^{n+3}\times\R^m \rightarrow \R$ as\footnote{See Remark \ref{rm:lie-derivative} for the notation for Lie derivative.}
\begin{align}
\label{eq:infgen_to_mf}
    \begin{split}
        D_\sprob (Z_t,U_t) 
        := & \ \gL_{\tilde{f}}\sprob(Z_t)+ \left( \gL_{\tilde{g}}\sprob(Z_t) \right) U_t \\ 
        &+\frac{1}{2}\text{tr} \left(\left[\tilde{\sigma}(Z_t)\right]\left[\tilde{\sigma}(Z_t)\right]^\intercal\Hess \sprob(Z_t)\right).
    \end{split}
\end{align}
From It\^o's Lemma,\footnote{It\^o's Lemma is stated as below: Given a $n$-dimensional real valued diffusion process $dX = \mu dt + \sigma dW$ and any twice differentiable scalar function $f: \R^n \rightarrow \R$, one has 
$
df= \left(\gL_\mu f + \frac{1}{2}\tr\left(\sigma\sigma^\intercal \Hess{f}\right)\right) dt + \gL_\mu \sigma dW.
$} the mapping \eqref{eq:infgen_to_mf} essentially evaluates the value of the infinitesimal generator of the stochastic process $Z_t$ acting on $\sprob$: \ie $A\sprob (Z_t) = D_\sprob (z,u)$ when the control action $U_t = u$ is used when $Z_t = z$.

We propose to constrain the control action $U_t$ to satisfy the following condition at all time $t$: 
\begin{align}
\label{eq:safety_condition_each_Zt}
	D_\sprob (Z_t,U_t) \geq -\alpha \left(\sprob(Z_t) - (1-\epsilon) \right).
\end{align}
Here, $\alpha: \R \rightarrow \R$ is assumed to be a monotonically-increasing, concave or linear function that satisfies $\alpha(0) \leq 0$. From \eqref{eq:infgen_to_mf}, condition \eqref{eq:safety_condition_each_Zt} is affine in $U_t$. This property allows us to integrate condition \eqref{eq:safety_condition_each_Zt} into a convex/quadratic program. 
\begin{theorem}
\label{lm:main_lemma}
   Consider the closed-loop system of $\eqref{eq:x_trajectory}$ and \eqref{eq:generic_controller}.\footnote{Recall from subsection~\ref{sec:SystemDescription} that $f$, $g$, $\sigma$, $N$, and $K_N$ are assumed to have sufficient regularity conditions.} Assume that $\sprob(z)$ in \cref{eq:cases_summary} is a continuously differentiable function of $z \in \R^{n+3}$ and $\mE [ \sprob(Z_t) ]$ is differentiable in $t$. If system \eqref{eq:x_trajectory} originates at $X_0=x$ with $\sprob(z)>1-\epsilon$, and the control action satisfies \eqref{eq:safety_condition_each_Zt} at all time, then the following condition holds:\footnote{Here, the expectation is taken over $X_t$ conditioned on $X_0 = x$, and $\sprob$ in \cref{eq:cases_summary} gives the probability of forward invariance/convergence of the future trajectories $\{X_\tau\}_{(t , t+T_t]}$ starting at $X_t$.}
	\begin{align}
	\label{eq:satisfy_control_policy}
		\mE\left[\sprob(Z_t) \right] \geq 1 - \epsilon
	\end{align}
	for all time $t \in \R_+$. 
\end{theorem}
\begin{proof}[Proof (\cref{lm:main_lemma})]
First, we show that
\begin{align}
\label{eq:expectation_less_than}
    \mE[\sprob(Z_\tau)] \leq 1-\epsilon
\end{align}
implies
\begin{align}
\label{eq:expectation_morethan_0}
    \mE \left[\alpha\left(\sprob(Z_\tau) - (1-\epsilon) \right) \right] \leq 0.
\end{align}
Let $\tau$ is the time when \eqref{eq:expectation_less_than} holds. We first define the 
events $D_i$ and a few variables $v_i,q_i, \text{ and } \delta_i$, $i\in\{0,1\}$, as follows:
\begin{align}
    \label{eq:D0}
    D_0 & = \left\{\sprob(Z_\tau) < 1-\epsilon \right\}, \\
    \label{eq:D1}
    D_1 & = \left\{\sprob(Z_\tau) \geq 1-\epsilon \right\}, \\
    \label{eq:V0}
    v_0 & = \mE \left[\sprob(Z_\tau) \mid D_0 \right] = 1-\epsilon-\delta_0, \\
    \label{eq:V1}
    v_1 & = \mE \left[\sprob(Z_\tau) \mid D_1 \right] = 1-\epsilon+\delta_1, \\
    \label{eq:P0}
    q_0 & = \mP(D_0), \\
    \label{eq:P1}
    q_1 & = \mP(D_1).
\end{align}
The left hand side of \cref{eq:expectation_less_than} can then be written as
\begin{align*}
    \mE[\sprob(Z_\tau)] & = \mE \left[\sprob(Z_\tau) \mid D_0 \right] \mP(D_0) +\mE \left[\sprob(Z_\tau) \mid D_1 \right] \mP(D_1) \\
    \label{eq:expectation_in_VP}
    & = v_0q_0 + v_1q_1. \eqnumber
\end{align*}
From
\begin{align}
\label{eq:fact1_from}
    \begin{aligned}
        \mE \left[\sprob(Z_\tau) \mid D_0 \right] & < 1-\epsilon, \\
        \mE \left[\sprob(Z_\tau) \mid D_1 \right] & \geq 1-\epsilon,
    \end{aligned}
\end{align}
we obtain
\begin{align}
\label{eq:fact1}
    \delta_0 \geq 0 \quad\text{and}\quad \delta_1 \geq 0.
\end{align}
Moreover, $\{q_i\}_{i\in\{0,1\}}$ satisfies
\begin{align}
\label{eq:fact2}
    \mP(D_0) + \mP(D_1) = q_0 + q_1 = 1.
\end{align}
Combining \cref{eq:expectation_less_than,eq:expectation_in_VP} gives
\begin{align}
\label{eq:combine_expectation}
    v_0q_0 + v_1q_1 \leq 1-\epsilon.
\end{align}
Applying \cref{eq:V0,eq:V1} to \cref{eq:combine_expectation} gives
\begin{align}
\label{eq:fact3_from}
    \left(1-\epsilon-\delta_0 \right)q_0 + \left(1-\epsilon+\delta_1 \right)q_1 \leq 1-\epsilon,
\end{align}
which, combined with \eqref{eq:fact2}, yields
\begin{align}
\label{eq:fact3}
    \delta_1 q_1 - \delta_0 q_0 \leq 0.
\end{align}
On the other hand, we have
\begin{align*}
    & \mE \left[\alpha\left(\sprob(Z_\tau) - (1-\epsilon) \right) \right] \\
    &\ \ \ \ \ = \ \mP(D_0) \left( \mE \left[\alpha\left(\sprob(Z_\tau) - (1-\epsilon) \right) \mid D_0 \right] \right) \\
    &\ \ \ \ \ \ \ \ + \mP(D_1) \left( \mE \left[\alpha\left(\sprob(Z_\tau) - (1-\epsilon) \right) \mid D_1 \right] \right) \eqnumber \\
    &\ \ \ \ \ = \ q_0 \left( \mE \left[\alpha\left(\sprob(Z_\tau) - (1-\epsilon) \right) \mid D_0 \right] \right) \\
    &\ \ \ \ \ \ \ \ + q_1 \left( \mE \left[\alpha\left(\sprob(Z_\tau) - (1-\epsilon) \right) \mid D_1 \right] \right) \label{eq:expectation_given_q} \eqnumber \\
    &\ \ \ \ \ \leq \ q_0 \left( \alpha\left(\mE \left[\sprob(Z_\tau) - (1-\epsilon) \mid D_0 \right] \right) \right) + \\
    &\ \ \ \ \ \ \ \ + q_1 \left( \alpha\left(\mE \left[\sprob(Z_\tau) - (1-\epsilon) \mid D_1 \right] \right) \right) \label{eq:inequality_from_jensen_rule} \eqnumber \\
    &\ \ \ \ \ = \ q_0 \left( \alpha\left(-\delta_0 \right) \right) + q_1 \left( \alpha\left(\delta_1 \right) \right) \label{eq:jense_inequality_given_V0V1} \eqnumber \\
    &\ \ \ \ \ \leq \ \alpha \left(-q_0\delta_0 + q_1\delta_1 \right) \label{eq:jensen_inequality_given_V0V1_assume_A2} \eqnumber \\
    &\ \ \ \ \ \leq \ 0. \label{eq:expectation_lessthan_0} \eqnumber
\end{align*}
Here, \cref{eq:expectation_given_q} is due to \cref{eq:P0,eq:P1}; \cref{eq:inequality_from_jensen_rule} is obtained from Jensen's inequality~\cite{Jensen1906} for concave function $\alpha$; \cref{eq:jense_inequality_given_V0V1} is based on \cref{eq:V0,eq:V1}; \cref{eq:jensen_inequality_given_V0V1_assume_A2} is given by assumption A2; and \cref{eq:expectation_lessthan_0} is due to \cref{eq:fact3}. Thus, we showed that \cref{eq:expectation_less_than} implies \cref{eq:expectation_morethan_0}.

Using Dynkin's formula, given a time-invariant control policy, the sequence $\mE[\sprob(Z_t)]$ takes deterministic value over time where the dynamics is given by
\begin{align}
    \frac{d}{d\tau} \mE[\sprob(Z_\tau)] = \mE [A\sprob(Z_\tau)].
\end{align}
Condition \cref{eq:safety_condition_each_Zt} implies
\begin{align}
\label{eq:dif_E_bigger_than_E}
    \mE [A\sprob(Z_\tau)] \geq - \mE \left[\alpha\left(\sprob(Z_\tau) - \left(1-\epsilon\right)\right)\right].
\end{align}
Therefore, we have
\begin{align}
\label{eq:dif_E_bigger_than_0}
        \frac{d}{d\tau} \mE [\sprob(Z_\tau)] \geq 0 \quad\text{whenever \(\mE[\sprob(Z_\tau)] \leq 1 - \epsilon\)}.
\end{align}
This condition implies
\begin{align}
    \mE[\sprob(Z_t)] \geq 1-\epsilon  \quad\text{for all \(t\in\R_+\).}
\end{align}
due to \cref{lm:lm2}, which is given below.
\end{proof}

\begin{lemma}
\label{lm:lm2} 
Let \(y\colon\R_+\to\R\) be a real-valued differentiable function that satisfies 
\begin{align}
\label{eq:y-ineq}
    {d\over dt}y_t \geq 0\quad\text{whenever \(y_t\leq L\)}.
\end{align}
Additionally, we assume \(y_0>L\).  Then
\begin{align}
    y_t\geq L\quad\text{for all \(t\in\R_+\).}
\end{align}
\end{lemma}
\begin{proof}
    Suppose there exists \(b\in\R_+\) such that \(y_b<L\). By the intermediate value theorem, there exists \(a\in (0,b)\) such that \(y_a = L\), and \(y_t<L\) for all \(t\in(a,b]\). Next, by the mean value theorem, there exists \(\tau\in (a,b)\) such that \((dy_t/dt)|_{t=\tau} = (y_b-y_a)/(b-a) < 0\). This contradicts condition \eqref{eq:y-ineq}.  
\end{proof}


\begin{corollary}
\label{crlry:safety_condition_for_I_II}
    Consider the closed-loop system of $\eqref{eq:x_trajectory}$ and \eqref{eq:generic_controller} with the assumptions stated in \cref{lm:main_lemma}. Let $\sprob$ be defined as type I or II in \cref{eq:cases_summary}.  
    If the system state originates at $X_0=x$ with $\sprob(z)>1-\epsilon$, and the control action satisfies \eqref{eq:safety_condition_each_Zt} at all time $t \in \R_+$, then condition \eqref{eq:FI-probability} holds. 
\end{corollary}
\begin{proof}[Proof (\cref{crlry:safety_condition_for_I_II})] From \eqref{eq:safe_prob_invariant}, \eqref{eq:cases_summary}, and \cref{lm:main_lemma}, we have
    \begin{align*}
    	& \mP \left(X_\tau \in \C(\varmargin_t), \forall \tau \in [t,t+T_t] \right) \\
    	&= \mE \left[\sprob(Z_\tau) \right]\\
    	&\geq 1 - \epsilon ,
    \end{align*}
    which yields \eqref{eq:FI-probability}.
\end{proof}

\begin{corollary}
\label{crlry:safety_condition_for_III_IV}
	Consider the closed-loop system of $\eqref{eq:x_trajectory}$ and \eqref{eq:generic_controller} with the assumptions stated in \cref{lm:main_lemma}. Let $\sprob$ be defined as type III or IV in \cref{eq:cases_summary}. 
	If the system state originates at $X_0=x$ with $\sprob(z)>1-\epsilon$, and the control action satisfies \eqref{eq:safety_condition_each_Zt} at all time $t \in \R_+$,
	then condition \eqref{eq:FC-probability} holds.
\end{corollary}
\begin{proof}[Proof (\cref{crlry:safety_condition_for_III_IV})] From \cref{eq:safe_prob_convergent,eq:cases_summary}, we have
    \begin{align*}
    	& \mP \left(\exists \tau \in [t,t+T_t] \text{ s.t. } X_\tau \in \gC(\varmargin_t) \right) \\
    	&= \mE \left[\sprob(Z_\tau) \right]\\
    	&\geq 1 - \epsilon ,
    \end{align*}
    which yields \cref{eq:FC-probability}.
\end{proof}

\subsection{Safe control algorithms}
\label{SS:Proposed_Algorithm}

Here, we propose two safe control algorithms based on the safety conditions introduced in subsection~\ref{SS:safety_conditions}. In both algorithms, the value of $\sprob$ is defined as type I or II in \cref{eq:cases_summary} when the safety specification is given as forward invariance condition, and as type III or IV when the safety specification is given as forward convergence condition.

\subsubsection{Additive modification}
\label{SSS:additive}

We propose a control policy of the form 
\begin{align}
\label{eq:additive_modification_policy}
    K_N(X_t, L_t, T_t)=N(X_t)+\addmodfunc(Z_t)(\gL_{\tilde{g}}\sprob(Z_t))^\intercal.
\end{align}
Here, $N$ is the nominal control policy defined in \eqref{eq:nominal_controller}. 

The mapping $\addmodfunc:\R^{n+3}\rightarrow\R_+$ is chosen to be a non-negative function that are designed to satisfy the assumptions of \cref{lm:main_lemma} and makes $U_t = K_N(X_t, L_t, T_t)$ to satisfy \eqref{eq:safety_condition_each_Zt} at all time. Then, the control action $U_t = K_N(X_t, L_t, T_t)$ yields
\begin{align}
\label{eq:additive_modification_policy-exp}
&\mE [ d\sprob(Z_t) ] 
= A \sprob(Z_t) \\
\nonumber
&=  \gL_{\tilde{f}}\sprob + (\gL_{\tilde{g}}\sprob) N  
+\addmodfunc \gL_{\tilde{g}}\sprob \left(\gL_{\tilde{g}}\sprob \right)^\intercal 
+
\frac{1}{2}\text{tr} \left( \tilde{\sigma}\tilde{\sigma}^\intercal\Hess \sprob \right).
\end{align}
As $\addmodfunc$ is non-negative, the term $\addmodfunc \gL_{\tilde{g}}\sprob \left(\gL_{\tilde{g}}\sprob \right)^\intercal $ in \eqref{eq:additive_modification_policy-exp} takes non-negative values. This implies that the second term additively modify the nominal controller output $N(X_t)$ in the ascending direction of the forward invariance probability \eqref{eq:FI-probability} or forward convergence probability \eqref{eq:FC-probability}. 

\subsubsection{Constrained optimization}
\label{SSS:conditioning}

We propose a control policy of the form 
\begin{align}
\label{eq:conditioning}
    \begin{aligned}
        K_N(X_t, L_t, T_t) = && \argmin_{u} & \ \ J(N(X_t),u)\\
        && \text{s.t.} & \ \ \eqref{eq:safety_condition_each_Zt},
    \end{aligned}
\end{align}
Here, $J:\R^m\times\R^m\rightarrow\R$ is an objective function that penalizes the deviation from the desired performance, the nominal control action, and/or the costs. It is also designed to satisfy the assumptions of \cref{lm:main_lemma} to comply with the safety specification \eqref{eq:FI-probability} or \eqref{eq:FC-probability}. The constraint of \eqref{eq:conditioning} imposes that \eqref{eq:safety_condition_each_Zt} holds at all time $t$, and can additionally capture other design restrictions.\footnote{For example, $K_N$ is Lipschitz continuous when $J(N(x),u) = u^\intercal H(x) u$ with $ H(x)$ being a positive definite matrix (pointwise in $x$).} When $\left( \gL_{\tilde{g}}\sprob(z) \right) \neq 0$ for any $z$, there always exists $u$ that satisfy the constraint \eqref{eq:safety_condition_each_Zt}.  


Both additive modification and conditioning structures are commonly used in the safe control of deterministic systems (see~\cite[subsection II-B]{chern2021safecontrol} and references therein). These existing methods are designed to find control actions so that the vector field of the state does not point outside of the safe set around its boundary. In other words, the value of the barrier function will be non-decreasing in the infinitesimal future outlook time horizon whenever the state is close to the boundary of the safe set. 
However, such myopic decision-making may not account for the fact that different directions of the tangent cone of the safe set may lead to vastly different long-term safety. In contrast, the proposed control policies \eqref{eq:additive_modification_policy} and \eqref{eq:conditioning} account for the long-term safe probability in $\sprob$, and are guaranteed to steer the state toward the direction with non-decreasing long-term safe probability when the tolerable long-term unsafe probability is about to be violated. When $\sprob$ is defined based on the closed-loop system involving \eqref{eq:x_trajectory} and \eqref{eq:nominal_controller}, its value can be computed offline. In such cases, the controller only needs to myopically evaluate the addition \eqref{eq:additive_modification_policy} or closed-form inequality conditions \eqref{eq:conditioning} in real time execution. In both cases, the computation efficiency is comparable to common myopic barrier function-based methods in a deterministic system.

\subsection{Improving the accuracy of gradient estimation}
\label{SS:Accurate_Gradient}


The safety condition \eqref{eq:safety_condition_each_Zt} requires us to evaluate $\sprob$, $\partial \sprob / \partial z$, and $\Hess \sprob$. These values can be estimated by applying Monte-Carlo methods on finite difference approximation formulas. However, for some systems and parameter ranges, naive sampling can produce noisy estimate of the probabilities and their gradients~\cite{chern2021safecontrolatcdc}. At a high spatial frequency, the randomness due to sampling can have a relatively larger impacts than the infinitesimal changes in the initial state. 

Such drawback in naive sampling can be complemented using additional information about the conditions that must be satisfied by the probabilities and their gradients. Here, we derive the safe/recovery probabilities as the solution to certain convection diffusion equations. The solution of the convection diffusion equations are guaranteed to be smooth and satisfy the neighbor relations of probabilities (see~\cite[Section 7.1]{evans1997partial} and reference therein for the regularity of convection diffusion). Such characterization allows the well-established numerical analysis techniques to be used to improve the accuracy of these estimates~\cite{leveque2002finite}.

Below, we present the convection diffusion equations. To emphasize the qualitatively different roles of $T_t$ and $(\varmargin_t,\phi(X_t),X_t)$, we introduce another state variable
    \begin{align}
        Y_t:=\begin{bmatrix}
    \varmargin_t \\
    \phi(X_t)\\
    X_t
    \end{bmatrix}\in\R^{n+2}.
    \end{align}

\begin{theorem}
\label{thm:pde_relations_theorem}
Let $S = \tilde{\sigma}\tilde{\sigma}^\intercal$. 
Let $\rho=\tilde{f}+\tilde{g}N$ if $\sprob$ in \eqref{eq:cases_summary} is defined for the closed-loop system of \eqref{eq:x_trajectory} and \eqref{eq:nominal_controller}, and $\rho=\tilde{f}+\tilde{g}K_N$ if $\sprob$ is defined for the closed-loop system of \eqref{eq:x_trajectory} and \eqref{eq:generic_controller}. The variable $\varsprob(Y_t,T_t):=\sprob(Z_t)$ for types I-IV satisfies the following convection diffusion equation~\cite[Theorems~1--4]{chern2021safecontrolatcdc}:
\begin{align}
\label{eq:pde_generic_form}
    \frac{\partial \varsprob}{\partial T} = \frac{1}{2}\grad\cdot(S\grad \varsprob) + \gL_{\rho - \frac{1}{2}\grad\cdot S}\varsprob,\ y[2]\geq y[1], T>0.
\end{align}
For types I and II, the boundary condition satisfies
\begin{align*}
    \label{eq:CauchyProblem_I_II}
        \begin{cases}
         \varsprob(y,T) = 0,& y[2]< y[1], T>0,\\
         \varsprob(y,0) = \mathbb{1}{\{y[2]\geq y[1]\}}(y),&y\in\R^{n+2}.
        \end{cases} \eqnumber
\end{align*}
For types III and IV, the boundary condition satisfies
\begin{align*}
    \label{eq:CauchyProblem_III_IV}
        \begin{cases}
         \varsprob(y,T) = 1,& y[2]< y[1], T>0,\\
         \varsprob(y,0) = \mathbb{1}{\{y[2]\geq y[1]\}}(y),&y\in\R^{n+2}.
        \end{cases} \eqnumber
\end{align*}
\end{theorem}
\vspace{5mm}

The methods to compute the values of $\sprob$, $\partial \sprob / \partial z$, and $\Hess \sprob$ are thorough and diverse. The characterization from Theorem~\ref{thm:pde_relations_theorem} can be exploited for improve the computation accuracy and efficiency. Examples of such techniques (non-mutually exclusive) are:
\begin{itemize}[leftmargin=*]
    \item Directly run Monte Carlo for neighboring states and approximate the gradient using finite difference methods.
    \item Evaluate the values of a boundary, and diffusing the boundary values to the interior/remaining areas. The boundary can be defined by the boundary condition given in \eqref{eq:CauchyProblem_I_II} or \eqref{eq:CauchyProblem_III_IV}. It can also be certain areas in \eqref{eq:pde_generic_form}, whose values can be evaluated using the MC method. 
    \item Use the relation in \eqref{eq:pde_generic_form} to derive the subspace that must be satisfied by $\sprob( z )$ and its neighbors $\sprob( z + \Delta z )$. This relation can be used to smooth out the results from the MC method: \eg the obtained probability can be projected onto the lower-dimensional subspace defined by \eqref{eq:pde_generic_form}. 
    \item Use condition \eqref{eq:pde_generic_form} to further derive the conditions that must be satisfied by $\partial \sprob / \partial z$ and $\Hess \sprob$. 
\end{itemize}

A review on the available methods and their tradeoffs is beyond the scope of this paper. The proposed approach do not constrain the computation of $\sprob$, $\partial \sprob / \partial z$, and $\Hess \sprob$ to be limited to any specific methods. 

\section{Example Use Case}
In this section, we show the efficacy of our proposed method in an example use case.
\subsection{Algorithms for comparison}
We compare our proposed controller with three existing safe controllers designed for stochastic systems. Below, we present their simplified versions.
\begin{itemize}[leftmargin=*]
    \item Proposed controller: The safety condition is given by
    \begin{align}
        D_\sprob(Z_t, U_t) \geq - \alpha (\sprob(Z_t) - (1-\epsilon)), 
    \end{align}
    where $\alpha > 0$ is a constant. We choose type I in \eqref{eq:cases_summary} with fixed time horizon and time-invariant zero margin, \ie $\mP \left(\minf_{X_t}(H) \geq 0 \right)$.
    \item Stochastic control barrier functions (StoCBF) \cite{clark2019control}: The safety condition is given by 
    \begin{align}
        \label{eq:StoCBF safe constraint}
        D_\phi(X_t, U_t) \geq - \eta \phi(X_t),
    \end{align}
    where $\eta > 0$ is a constant.
    Here, the mapping $D_\phi:\R^{n}\times\R^m \rightarrow \R$ is defined as the infinitesimal generator of the stochastic process $X_t$ acting on the barrier function $\phi$, \ie
    \begin{align}
    \label{eq:infgen_to_phi}
        \begin{split}
            D_\phi (X_t,U_t) := &A\phi (X_t) \\
            = &\gL_{f}\phi(X_t)+\gL_{g}\phi(X_t) U_t \\ 
            &+\frac{1}{2} \text{tr} \left(\left[\sigma(X_t)\right]\left[\sigma(X_t)\right]^\intercal\Hess \phi(X_t)\right).
        \end{split}
    \end{align} 
    
    This condition constrains the average system state to move within the tangent cone of the safe set. 
    \item Probabilistic safety barrier certificates (PrSBC) \cite{luo2019multi}: The safety condition is given by
    \begin{align}
        \label{eq:PrSBC safe constraint}
        \textbf{P}\left(D_\phi(X_t, U_t) + \eta \phi(X_t) \geq 0\right) \geq 1 - \epsilon,
    \end{align}
    where  $\eta > 0$ is a constant. This condition constrains the state to stay within the safe set in the infinitesimal future interval with high probability. 
    \item Conditional-value-at-risk barrier functions (CVaR) \cite{ahmadi2020risk}: The safety condition is given by
    \begin{align}
        \label{eq:CVaR safe constraint}
        \text{CVaR}_\beta \left(\phi(X_{t_{k+1}})\right) \geq \gamma \phi(X_{t_k})
    \end{align}
    where $\gamma \in (0,1)$ is a constant, $\{t_0 = 0, t_1, t_2,\cdots\}$ is a discrete sampled time of equal sampling intervals. 
    This is a sufficient condition to ensure the value of $\text{CVaR}^{k}_\beta(\phi(X_{t_k}))$ conditioned on $X_0 = x$ to be non-negative at all sampled time $t_{k \in \mathbb{Z}_+}$. The value of $\text{CVaR}^{k}_\beta(\phi(X_{t_k}))$ quantifies the evaluation made at time $t_0 = 0$ about the safety at time $t_k$.
\end{itemize}

\subsection{Settings}
We consider the control affine system (\ref{eq:x_trajectory}) with $f(X_t) \equiv A = 2$, $g(X_t) \equiv 1$, $\sigma(X_t) \equiv 2$. 
The safe set is defined as
\begin{align}
\begin{split}
    \gC(0) = \left\{x \in \R^n : \phi(x) \geq 0 \right\}, 
\end{split}
\end{align}
with the barrier function $\phi(x) := x-1$. The safety specification is given as the forward invariance condition. 
The nominal controller is a proportional controller $N(X_t) = -K X_t$ with $K = 2.5$. The closed-loop system with this controller has an equilibrium at $x=0$ and tends to move into the unsafe set in the state space.
We consider the following two settings: 
\begin{itemize}[leftmargin=*]
    \item \textbf{Worst-case safe control:} We use the controller that satisfies the safety condition with equality at all time to test the safety enforcement power of these safety constraint. Such control actions are the riskiest actions that are allowed by the safety condition. The use of such control actions allows us to evaluate the safety conditions separated from the impact of the nominal controllers. Here we want to see whether our proposed controller can achieve non-decreasing expected safety as intended.
    \item \textbf{Switching control:} We impose safe controller only when the nominal controller does not satisfy the safety constraint. Here we want to see how the proposed controller performs in practical use, where typically there is a control goal that is conflicting with safety requirements.
\end{itemize}

We run simulations with $d t = 0.1$ for all controllers. The initial state is set to $x_0 = 3$. For our controller, each Monte Carlo approximation uses $10000$ sampled trajectories. The parameters used are listed in Table \ref{tb:parameter list}. Since the parameter $\alpha$ in the proposed controller has a similar effect as $\eta$ in StoCBF and PrSBC, we use the same values for these parameters in those controllers. The parameter $\epsilon$ is the tolerable probability of unsafe events both in the proposed controller and PrSBC, so we use the same values of $\epsilon$ for both algorithms for a fair comparison. 

\begin{table}
\caption{parameters used in simulation}
\vspace{-1mm}
\label{tb:parameter list}
\begin{center}
\begin{tabular}{ |c|c| } 
\hline
\textbf{Controller} & \textbf{Parameters}  \\
\hline
\hline
Proposed controller & $\alpha = 1$, $\epsilon = 0.1$, $H=10$ \\
\hline
StoCBF & $\eta = 1$ \\
\hline
PrSBC & $\eta = 1$, $\epsilon = 0.1$ \\
\hline
CVaR & $\gamma = 0.65$, $\beta = 0.1$ \\
\hline
\end{tabular}
\end{center}
\end{table}


\subsection{Results}
Fig.~\ref{fig:worst_case} shows the results in the worst-case setting. The proposed controller can keep the expected safe probability $\mE[\sprob(X_t)]$ close to $0.9$ all the time, while others fail to keep it at a high level with used parameters. A major cause of failure is due to the accumulation of rare event probability, leading to unsafe behaviors. This shows the power of having a provable performance for non-decreasing long-term safe probability over time. For comparable parameters, the safety improves from StoCBF to PrSBC to CVaR. This is also expected as constraining the expectation has little control of higher moments, and constraining the tail is not as strong as constraining the tail and the mean values of the tail. 

\begin{figure}[hptb]
	\centering
	\begin{subfigure}{.2395\textwidth}
	    \includegraphics[width=\textwidth]{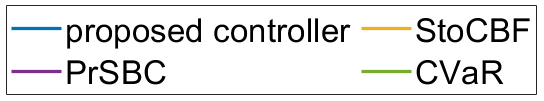}
	\end{subfigure}
	
	\begin{subfigure}{.2395\textwidth}
		\includegraphics[width=\textwidth]{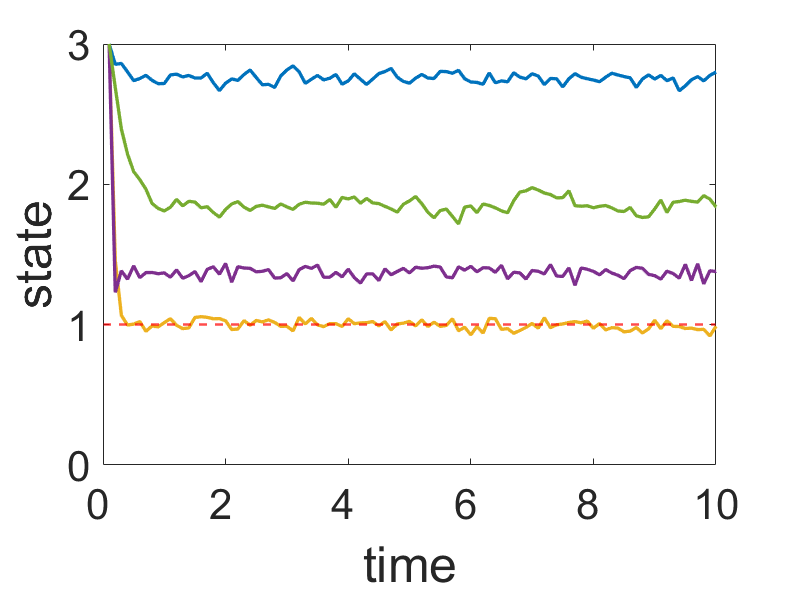}
        \caption{average state}
		\label{fig:worst-case state}
	\end{subfigure}
	\begin{subfigure}{.2395\textwidth}
		\includegraphics[width=\textwidth]{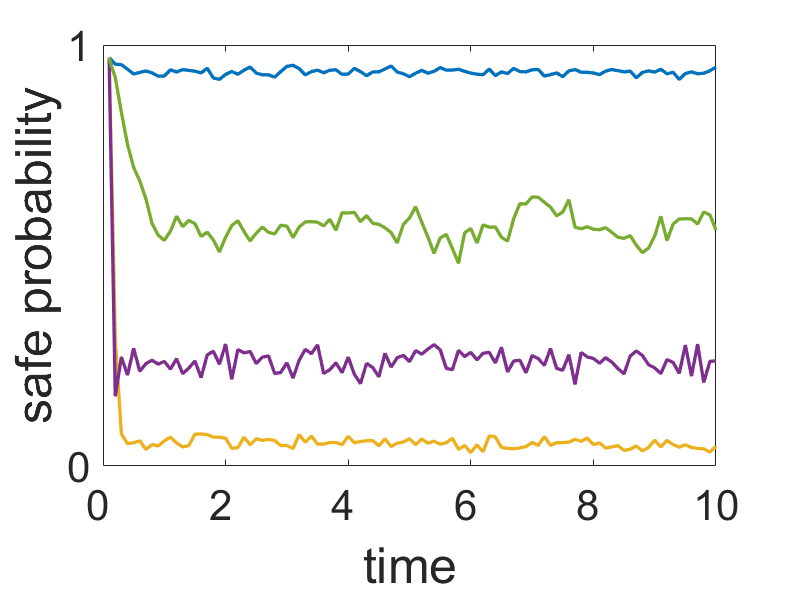}
        \caption{expected safe probability}
		\label{fig:worst-case safe prob}
	\end{subfigure}
	\caption{Results in the worst-case setting where \textbf{(a)} shows the average system state over 50 trajectories and \textbf{(b)} shows the expected safe probability \eqref{eq:FI-probability}.}
	\label{fig:worst_case}
\end{figure}

Fig.~\ref{fig:switch control} shows the results in the switching control setting. We obtained the empirical safe probability by calculating the number of safe trajectories over the total trials. In this setting, the proposed controller can keep the state within the safe region with the highest probability compared to other methods, even when there is a nominal control that acts against safety criteria. This is because the proposed controller directly manipulates dynamically evolving state distributions to guarantee non-decreasing safe probability when the tolerable unsafe probability is about to be violated, as opposed to when the state is close to an unsafe region. Our novel use of forward invariance condition on the safe probability allows a myopic controller to achieve long-term safe probability, which cannot be guaranteed by any myopic controller that directly imposes forward invariance on the safe set.


\begin{figure}[hptb]
	\centering
	\begin{subfigure}{.2395\textwidth}
	    \includegraphics[width=\textwidth]{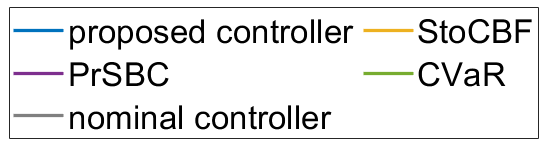}
	\end{subfigure}
	
	\begin{subfigure}{.2395\textwidth}
		\includegraphics[width=\textwidth]{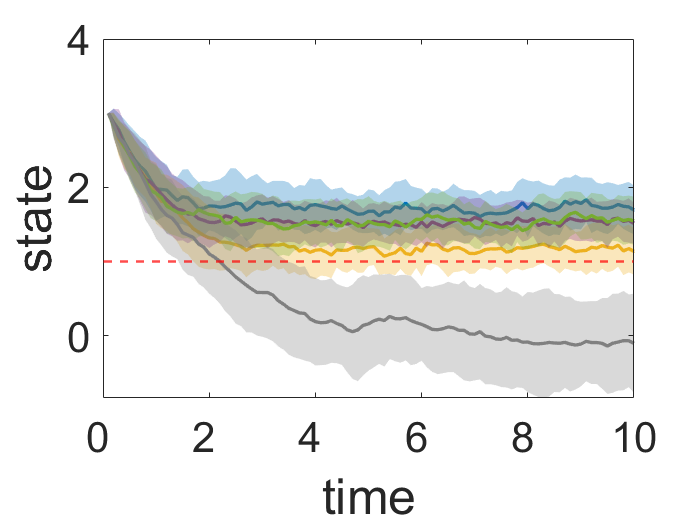}
		\caption{average state}
		\label{fig:switching state}
	\end{subfigure}
	\begin{subfigure}{.2395\textwidth}
		\includegraphics[width=\textwidth]{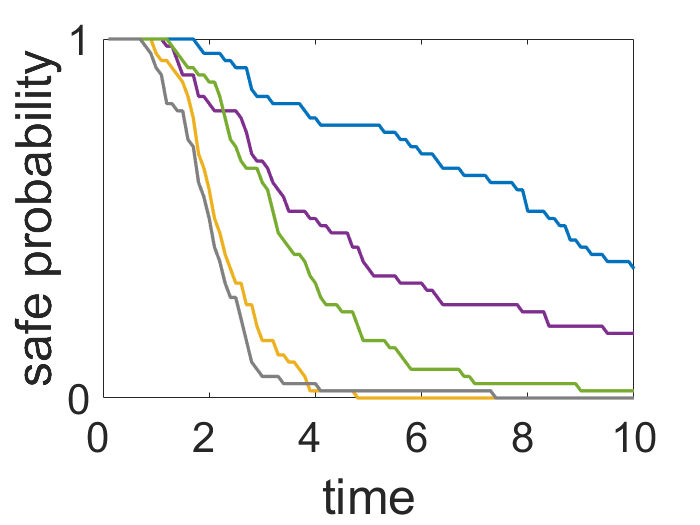}
		\caption{empirical safe probability}
		\label{fig:switching safe prob}
	\end{subfigure}
	\caption{Results in the switching control setting where \textbf{(a)} shows the averaged system state of 50 trajectories with its standard deviation and \textbf{(b)} shows the empirical safe probability.}
	\label{fig:switch control}
\end{figure}

\section{Conclusion}

In this paper, we considered the problem of ensuring long-term safety with high probability in stochastic systems. We proposed a sufficient condition to control the long-term safe probability of forward invariance (staying within the safe region) and forward convergence (recovering to the safe region). We then integrated the proposed sufficient condition into a myopic controller which is computationally efficient. We additionally outline possible techniques to improve the computation accuracy and efficiency in evaluating the sufficient condition. 
Finally, we evaluated the performance of our proposed controller in a numerical example. Although beyond the scope of this paper, the proposed framework can also be used to characterize the speed and probability of system convergence and may be useful in finite-time Lyapunov analysis in stochastic systems.

\end{document}